\documentclass[12pt]{amsart}

\usepackage{amsmath, amsthm, amsfonts, amssymb, enumitem}
\usepackage{mathrsfs}

\usepackage[margin=29mm]{geometry}

\newtheorem{proposition}{Proposition}
\newtheorem{theorem}{Theorem}

\newtheorem{lemma}{Lemma}
\newtheorem{rk}{Remark}
\newtheorem{definition}{Definition}

\newcommand{\re}{\mathbb{R}}

\newcommand{\p}{\mathbb{P}}

\newcommand{\grad}{\nabla}

\def\esp{\mathbb{E}}
\def\bN{\mathbb{N}}

\newcommand{\wh}{\mathscr{W}}

\newcommand{\bW}{\overline{W}}

\newcommand{\field}{\mathcal{X}}
\newcommand{\mart}{\mathcal{M}}
\newcommand{\spart}{\mathcal{S}}
\newcommand{\apart}{\mathcal{B}}
\newcommand{\mpart}{\widetilde{\mathcal{B}}}

\newcommand{\fieldn}{\mathcal{X}^n}

\newcommand{\gradn}{\nabla^n}
\newcommand{\lapln}{\Delta^n}

\newcommand{\cale}{\mathcal{E}}

\def\Vvv{{\rm\mathbb{V}ar}}  

\def \simless {\mathbin{\lower 3pt\hbox{$\rlap{\raise 5pt
              \hbox{$\char'074$}}\mathchar"7218$}}}

\author[Milton Jara and Gregorio R. Moreno Flores]{Milton Jara$^1$ and Gregorio R. Moreno Flores$^2$}

\thanks{ AMS 2010 {\it subject classifications}. Primary  60H15
 secondary 60K35, 82B44
 }


\thanks{$^1$ Instituto de Matem\'atica Pura e Aplicada. Partially supported by CNPq and FAPERJ}

\thanks{$^2$ Pontificia Universidad Cat\'olica de Chile. Partially supported by Fondecyt grant 1171257 and N\'ucleo Milenio `Modelos Estoc\'asticos de Sistemas Complejos y Desordenados'}

\thanks{This project has received funding from the European Research Council  
(ERC) under the European Union’s Horizon 2020 research and innovative  
programme (grant agreement No 715734), and from MATH Amsud `Random Structures and Processes in Statistical Mechanics'}

\address[Milton Jara]{Instituto de Matem\'atica Pura e Aplicada, Estrada Dona Castorina 110, 22460320 Rio de Janeiro, Brazil}

\email{mjara@impa.br}

\address[Gregorio R. Moreno Flores]{Facultad de Matem\'aticas\\
Pontificia Universidad Cat\'olica de Chile\\
Vicu\~na Mackenna 4860, Macul\\
Santiago, Chile}

\email{grmoreno@mat.uc.cl}

\email{}

\date{}

\title[Stationary directed polymers]{Stationary Directed Polymers and Energy Solutions of the Burgers Equation}

\begin{document}
	
\begin{abstract} 
	We consider the stationary O'Connell-Yor model of semi-discrete directed polymers in a Brownian environment in the intermediate disorder regime and show convergence of the increments of the log-partition function to the energy solutions of the stochastic Burgers equation. 
	
	The proof does not rely on the Cole-Hopf transform and avoids the use of spectral gap estimates for the discrete model. The key technical argument is a second-order Boltzmann-Gibbs principle.
\end{abstract}

\maketitle


\section{Introduction, Model and Results}

\subsection{KPZ equation and Stochastic Burgers equation}

The Kardar-Parisi-Zhang equation \cite{KPZ}, or KPZ equation, was introduced in the physics literature as a model for interface motions in generic situations. The typical physical set-up is the following: suppose we have a thin physical system where a stable and a meta-stable phase can coexist and suppose both phases are separated by an interface. We are concerned with the behaviour of such an interface as the stable phase invades the meta-stable region. The first thing one observes is a net motion of the interface, meaning that it has in average a non-zero velocity. At a closer look, we can observe very intricate fluctuations with an atypical order of magnitude of $t^{1/3}$ and highly non-Gaussian statistics. Assuming that the position of the interface can be (locally) described by a height function $h(t,x)$, the authors of \cite{KPZ} conclude that its dynamics is governed by the equation
\begin{eqnarray*}
	\partial_t h &=& \nu \partial^2_x h + \lambda \left| \partial_x h\right|^2 + \sqrt{D} \wh,
\end{eqnarray*}
where $\wh$ is a space-time white noise and $\nu$, $\lambda$ and $D$ are constants that depend on the precise model under consideration. In particular, the quantity $\sqrt{D}$ represents the intensity of the noise which, in the terminology of \cite{SQ}, represents the random back and forth between the two phases.

Perhaps the most accurate experimental realization of this dynamics is given by the growing interfaces in liquid crystal turbulence (see \cite{liquid} and references therein). There, a thin film of turbulent liquid crystal is kept out of thermal equilibrium. Then, a seed of the stable phase is created and grows as a cluster. The statistics of the fluctuation of the interface separating the two phases match the theoretical predictions with spectacular accuracy.

\vspace{1ex}

The KPZ equation constitutes a particular representative of a huge family of models known as the KPZ universality class. These models are characterized by displaying cube-root fluctuations which laws rescale to non-Gaussian distributions that first appeared in random matrix theory. Despite their complicated nature, some of these models have even explicit laws  that allow for fine asymptotic analysis. The study of these models  has generated an huge body of work in the mathematics and physics communities  that is impossible to summarize in a concise way. We refer the reader to the reviews \cite{FS}, \cite{C-review} and \cite{Q-review} for a recent exposition of the state of the art.

Here, we will not be concerned with the `integrable' nature of the KPZ universality class but will instead focus on a particular model which, in a very precise regime, rescales to the KPZ equation. The emergence of KPZ as scaling of discrete models first appeared in the work \cite{BG} for the weakly asymmetric exclusion process, the main representative of the so-called \textit{weakly asymmetric limits}. This type of limit has since then appeared in many contexts, for instance \cite{ACQ,GJ-arma,GJSeth,JM} among others. Our setting is an example of a different kind of limit, the \textit{intermediate disorder regime}, first observed in \cite{AKQ} (see also \cite{MQR,CG}).

\vspace{1ex}

The KPZ equation has two `avatars': the stochastic Burgers equation and the stochastic heat equation. Letting $u=\partial_x h$ be the slope of $h$, we can see that $u$ satisfies the equation
\begin{eqnarray}\label{eq:Burgers}
	\partial_t u = \nu \partial_x^2 u + \lambda \partial_x u^2 + \sqrt{D} \partial_x \wh.
\end{eqnarray}
This is known as the stochastic Burgers equation. The notion of solutions for the KPZ and Burgers equation is mathematically very delicate.  In the KPZ equation, the best space regularity one can expect is that of a Brownian motion. As such, its first derivative is a genuine distribution and its square requires a very careful treatment to be properly defined. Burgers equation of course shares a similar problems and is in fact distribution-valued.

An early solution to this issue consisted in taking a clever non-linear transform of $h$ which removes the non-linear part of the equation: let $\displaystyle Z(t,x)=\exp\left\{\frac{\lambda}{\nu}\log h(t,x)\right\}$. Then, $Z$ satisfies the equation
\begin{eqnarray}\label{eq:SHE}
	\partial_t Z = \nu \partial_x^2 Z + \frac{\lambda \sqrt{D}}{\nu} Z \wh,
\end{eqnarray}
known as the stochastic heat equation (SHE). This equation can be solved by ad-hoc methods and, then, the solution to KPZ can be `defined' as $\displaystyle h(t,x) = \frac{\nu}{\lambda} \log Z(t,x)$. This is the so-called Cole-Hopf solution to the KPZ equation which can be traced back at least to \cite{BG}. Although this provides a notion of solution which is useful for many purposes (such as showing the convergence of a wide family of discrete models), it is unsatisfactory in the sense that it does not show that $h$ actually satisfies an equation. We will come back to the Cole-Hopf solution in the next section as it provides a natural link between KPZ/Burgers/SHE and directed polymers.

In recent years, more robust theories of existence and uniqueness of KPZ/Burgers equation have emerged. The first one \cite{H-KPZ} was pioneered by M. Hairer and lead to the development of the theory of regularity structures \cite{H-RS}. This theory allows to give a solid notion of solution for a wide family of singular stochastic PDEs as well as providing a framework to prove the convergence of discrete models \cite{HM, CM}.  

This breakthrough was shortly followed by an existence and uniqueness theory for KPZ/Burgers in the framework of paracontrolled distributions \cite{GP-reloaded}. As in the case of regularity structures, this theory can be used to treat other stochastic PDEs beyond KPZ \cite{CC} and is amenable to show the convergence of discrete models \cite{CGP, MP}. Furthermore, this theory can be succesfully applied to the KPZ equation on the whole real line \cite{PR}. 

A third approach is provided by the theory of energy solutions introduced in \cite{GJ-arma} and further developed in \cite{GuJ}. In this approach, Burgers equation is formulated as a martingale problem. Uniqueness for this weak formulation on the whole line was proved in \cite{GP-uniqueness}. Since \cite{GJ-arma}, this approach was successfully applied to show the convergence of many discrete models to the KPZ/ Burgers equation \cite{GJSeth, GJS, DGP, JM}. One substantial advantage is that it requires very weak quantitative estimates. So far, this theory is mainly restricted to the stationary setting (see however \cite{GP-noneq, GP-generator}). Let us go back to \eqref{eq:Burgers}. If $\lambda=0$, the equation becomes a linear stochastic heat equation with additive noise. Its solutions are explicit, Gaussian and have the spatial white noise as an invariant measure. One remarkable fact is that this invariant measure is preserved by the addition of the non-linear term. In the KPZ setting, the white noise initial condition corresponds to a double-sided Brownian motion. Of course, even if at fixed times the spatial distribution of the process is fairly simple, the time-correlations are extremely complicated. This will be the context considered in this work.

\subsection{The Cole-Hopf transformation, the Stochastic Heat Equation and Directed Polymer}

The Cole-Hopf transform provided a way to give a meaning to the solutions of the KPZ equation in terms of the solutions of the SHE. For the SHE, solutions can be described by means of a chaos expansion, taking advantage of the particular structure of the equation \cite{Q-review}. This also provided a starting point to show the convergence of discrete models which themselves satisfy a discrete version of the Cole-Hopf transform, for instance the exclusion process \cite{BG, ACQ} (for which the discrete Cole-Hopf transform dates back to \cite{Gartner}), directed polymers \cite{AKQ, MQR} which are naturally formulated at the level of the SHE, one-dimensional random walks in vanishing random environments \cite{CG} and weakly asymmetric bridges \cite{L} among others. Once again, we point the reader to the reviews \cite{C-review} and \cite{Q-review} for a more detailed summary of the field. 

This approach has two limitations. First, as it actually proves convergence to the SHE, it does not in principle provide convergence to KPZ/Burgers (regardless of the way these are interpreted). For instance, it does not allow to obtain direct convergence to Burgers equation for the occupation field of the exclusion process as in \cite{GJ-arma} although it can be used to show convergence of its height function to the Cole-Hopf solution of KPZ. Second, this approach relies heavily on the availability of a discrete Cole-Hopf transform which is not available for relevant models such as the Sasamoto-Spohn model \cite{SS} or the coupled diffusions considered in \cite{DGP}.

These difficulties can be circumvented using the theories described in the previous section. The works \cite{HM, CM} provide a general framework to treat discrete models in the context of regularity structures. In the framework of paracontrolled distributions, the work \cite{GP-reloaded} was successful in giving a first proof of the convergence of the periodic Sasamoto-Spohn model and the works \cite{CGP,MP} developed robust arguments to treat stochastic PDEs on the lattice. In the context of energy solutions, speed-change exclusion dynamics was treated in \cite{GJ-arma, GJSeth}, interacting diffusions in \cite{DGP} and the Sasamoto-Spohn model on the whole line in \cite{JM}. In all these last examples, direct convergence to Burgers equation for the fluctuation field is proved.

\vspace{1ex}

One interesting fact about the Cole-Hopf solution is that it gives a direct link between directed polymers and the KPZ/SHE equation. Consider the SHE \eqref{eq:SHE} with initial condition $Z(0,\cdot)=\delta_0$ and assume for a moment that the potential $\wh$ is smooth. Then, Feynman-Kac formula yields the explicit solution (with $D=1$ and $\lambda=\nu=\frac12$)
\begin{eqnarray}\label{eq:FK}
	Z(t,x) = \mathbf{E}_x\left[ e^{\int^t_0 \wh(t-s,b_s)\,ds}\right],
 \end{eqnarray}
 where $\mathbf{E}_x$ is the expectation with respect to the law of a Brownian bridge $(b_s:\, 0\leq s \leq t)$ with $b_0=0$ and $b_t=x$. As such, $Z$ can be viewed as the partition function of a directed polymer model where the energy of a path is given by $H(b)=\int^t_0 \wh(s,b_s)\,ds$. These directed polymers in random environment were introduced in the physics literature as a model for the roughening of interfaces in random environment \cite{Henley-Huse}. They have been the object of a vast body of mathematical work since then (see \cite{Comets-book} for a recent monograph on the subject). When $\wh$ is taken as a white noise, it is possible to give a sense to \eqref{eq:FK} (see \cite{Q-review}) and even to a \textit{continuum polymer measure} \cite{AKQ2}.


\subsection{Semi-discrete Directed Polymers in a Brownian Environment and the Main Result}

%

Polymer models are defined by specifying a path space and an environment. We will work exclusively with the model of Directed Polymers in a Brownian Environment introduced in \cite{OY}. In this case:

\begin{itemize}
	\item {\bf Polymer paths} are nondecreasing c\`adl\`ag paths  $x: [0,t]\to \bN$ with nearest-neighbor jumps,   $x(0)=1$, and $x(t)=n$. A path can be coded in terms of its jump times $0=s_0<s_1<\cdots<s_{n-1}<s_n=t$. 
  
\smallskip
  
	\item {\bf The environment} consists of a family of independent double-sided one-dimensional standard Brownian motions $\{B^{(j)}(\cdot): j\geq 1\}$ with  $B^{(j)}(0)=0$. 
\end{itemize}
At level $j$, the path  collects the increment $B^{(j)}(s_{j-1},s_j)=B^{(j)}(s_j)-B^{(j)}(s_{j-1})$. The partition function in a fixed  Brownian environment is  defined as
\begin{eqnarray} \nonumber
Z_{\beta}(t,n)&=&    \hspace{-15pt}
\int\limits_{0<s_{1}<\dotsm<s_{n-1}<t} \hspace{-15pt} \exp {\beta}\bigl[ B^{(1)}(0,s_1) +B^{(2)}(s_1,s_{2})
+\dotsm + B^{(n)}(s_{n-1},t)\bigr] \,ds_{1,n-1} .  
\end{eqnarray}
for $(n,t)\in\bN\times[0,\infty)$, where $ds_{1,n-1}$ is a short-hand notation for $ds_1\dotsm ds_{n-1}$.

We will be mainly concerned with a \textit{stationary} version of the model: enlarge the environment by adding another  Brownian motion $B^{(0)}$ independent of $\{B^{(j)}:\, j\geq 1\}$ and introduce a new parameter
$\theta\in(0,\infty)$. The stationary partition function is defined as
\begin{eqnarray*}
\begin{aligned}   Z_{\beta,\theta}(t,n) &=\hskip-20pt 
\int\limits_{-\infty<s_{0}<s_1<\dotsm<s_{n-1}<t}  \hskip-30pt \exp\bigl[ \beta B^{(0)}(s_0)+\theta s_0  +  \beta \left\{B^{(1)}(s_0,s_1) 
+\dotsm + B^{(n)}(s_{n-1},t)\right\}\bigr] \,ds_{0,n-1},
\end{aligned}\label{Zdef2}
\end{eqnarray*}
for $n\geq 1$ and $Z_{\beta,\theta}(t,0)=\exp\{\beta B^{(0)}(t)+\theta t\}$.
Here, stationarity refers to a specific property of the model highlighted in \cite{OY}: define
\begin{eqnarray*}
	u_{\beta,\theta}(t,j) = \log Z_{\beta,\theta}(t,j)-\log Z_{\beta,\theta}(t,j-1),\quad j\geq 1.
\end{eqnarray*}
Then, for all $t\geq 0$, $\{u_{\beta,\theta}(t,j):\, j\geq 1\}$ is an i.i.d. family with
\begin{eqnarray*}
	e^{-u_{\beta,\theta}(t,j)} \sim \beta^2\text{Gamma}(\beta^{-2}\theta).
\end{eqnarray*}
In other words, if we denote by $m_{\beta,\theta}$ the law of $-\log X-\log \beta^2$ where $X\sim\text{Gamma}(\beta^{-2}\theta)$ and $\mu_{\beta,\theta} = m_{\beta,\theta}^{\otimes \mathbb{N}}$, then $\mu_{\beta,\theta}$ is the stationary measure of the process $\{u_{\beta,\theta}(\cdot,j):\, j\geq 1\}$.

The processes $Z_{\beta,\theta}$, $h_{\beta,\theta}=\log Z_{\beta,\theta}$ and $u_{\beta,\theta}$ can be seen as semi-discrete stationary versions of the stationary SHE, KPZ and Burgers equations respectively. The link between $Z_{\beta,\theta}$ and the SHE can be seen as a rigorous version of the Feynman-Kac formula \eqref{eq:FK}. It turns out that these processes actually satisfy lattice versions of these equations. In the case of Burgers, an application of It\^o's formula shows that $u_j(\cdot)=u_{\beta,\theta}(\cdot,j)$ satisfies the system of stochastic differential equations
\begin{eqnarray*}
	du_j = \left( e^{-u_{j}}-e^{-u_{j-1}}\right)\, dt + \beta \left( dB_t^{(j)}-dB_t^{(j-1)}\right)
\end{eqnarray*}
A naive Taylor expansion suggests that $e^{-u_{j}}-e^{-u_{j-1}}\simeq (u_{j-1}-u_j) + \frac12 (u_j^2 - u_{j-1}^2)$. In an appropriate skew scaling, the discrete gradient $u_{j-1}-u_j$ becomes a second derivative while the difference of squares is reminiscent of the term $\partial_x u^2$ in \eqref{eq:Burgers}. We will show that this is indeed the case although it does not follow from such a simple argument. It is actually the core of our proof. The system above was already observed in \cite{S}. Note that complicated non-linearities leading to Burgers equation where considered in the works \cite{HQ,HX} and later in \cite{GP-HQ} in the context of energy solutions. The works \cite{HX} and \cite{GP-HQ} deal with general  examples of \textit{weakly asymmetric scaling}. Our work is an example of \textit{intermediate disorder scaling}. The reference \cite{HQ} deals with both settings.  One advantage of the energy solution approach is that it allows to consider models on the whole real line although only in equilibrium.

\vspace{1ex}

We can state our result: 
denote $u_{\beta,\theta}$ by $u^n$ whenever $\beta=n^{-1/4}$ and $\theta = 1 + \frac{1}{2\sqrt{n}}$ with stationary initial condition. 
We fix once for all an increasing diverging sequence $(a_n)_n$ such that $\displaystyle \lim_{n\to\infty} \frac{a_n}{\sqrt{n}}=0$.
Let $\rho_n:=\esp[u^n]$ and define the fluctuation field $\fieldn$ acting on test functions by
\begin{eqnarray*}
	\fieldn_t(\varphi) = \sum_j (u^n(tn,j)-\rho_n) \varphi(\tfrac{j-nt-a_n \sqrt{n}}{\sqrt{n}}).
\end{eqnarray*}

\begin{theorem}
The sequence of processes $(\field^n_t:\, t\in[0,T])_{n\geq 1}$ converges in distribution in $C([0,T],\mathcal{S}'(\re))$
to the unique energy solution of the Burgers equation
\begin{eqnarray}\label{eq:discrete-Burgers}
	\partial_t u = \frac12 \partial^2_x u + c \partial_x u -\frac{1}{2} \partial_x u^2 + \partial_x{\wh},
\end{eqnarray}
where $c\in\re$ is an explicit constant and $\wh$ is a space-time white noise.
\end{theorem}
\begin{rk}\rm
	It is reasonable to expect a non-trivial transport term as there are several sources of asymmetry in the model. First, our scaling was meant to properly normalize the partition function. As such, its logarithm is slightly off-center. Second, the direction $(1,1)$ is not exactly the characteristic direction or, equivalently, to force it to be characteristic, a more careful choice of constants has to be made (see \cite{SV,MSV}). Of course, both settings are asymptotically equivalent. Finally, this model can be seen as a system of coupled diffusions in the highly non-symmetric potential  
	$V(u)=e^{-u}-1+u$. This is another source of asymmetry. The precise value of $c$ can be obtained by careful book-keeping along the proof. We found it to be $-\frac{9}{10}$.  In any case, this transport term can be removed by a change of coordinates in the equation and, at the discrete level, with a more careful centering of the test functions.
\end{rk}

\begin{rk}\rm
	The sequence $(a_n)_n$ is introduced to deal with the fact that the discrete model is defined only for $j\geq 1$. For any compactly supported test function, the fluctuation field will then be well defined for $n$ large enough. As the system is stationary, this correction is harmless.
\end{rk}

\begin{rk}\rm
	We note that the result above can be easily generalized to systems of SDEs of the type
	\begin{eqnarray*}
	du_j = \left( V'(u_{j-1})-V'(u_j)\right)\, dt + \beta \left( dB_t^{(j)}-dB_t^{(j-1)}\right),
\end{eqnarray*}
where $V$ is a real-valued function which is quadratic at $0$ and has appropriate growth at $\pm \infty$, provided the dynamics above can be properly defined. The existence of the dynamics is a difficult question (see \cite{MO} for results in this direction). In our case (where $V(u)=e^{-u}-1+u$), the interpretation of the system \eqref{eq:discrete-Burgers} in terms of directed polymers settles the issue but, in general, this connection is lost. We will not consider such a general framework in this article.
\end{rk}

\subsection{Structure of the Article}

In Section 2, we recall the notion of energy solutions of Burgers equation. In Section 3, we carefully state the system of SDEs satisfied by the model, identify its different components and give a martingale interpretation. In Section 4, we present some useful estimates on the  moments of the discrete model. In Section 5, we prove the dynamical estimates which are the core of our proof. In particular, we prove the second order Boltzmann-Gibbs principle in Section 5.2. In Section 6, we prove the tightness of the fluctuation field and identify its limit in Section 7. The Appendix contains additional estimates needed in Sections 6 and 7.


\subsection{General Notations}
Recall that we fixed an increasing diverging sequence $(a_n)_n$ such that $\displaystyle \lim_{n\to\infty} \frac{a_n}{\sqrt{n}}=0$.
For test functions $\varphi$, we define
\begin{eqnarray*}
	\varphi^n_j &=& \varphi(\tfrac{j-nt-a_n\sqrt{n}}{\sqrt{n}}),\\
	\gradn \varphi^n_j &=& \frac{\sqrt{n}}{2}(\varphi^n_{j+1}-\varphi^n_{j-1})\\
	\lapln\varphi^n_j &=& \frac{n}{2} (\varphi^n_{j+1}+\varphi^n_{j-1}-2\varphi^n_j).
\end{eqnarray*}
Note that, even though the discretization depends on the value of $t$, we remove it from the notation as no confusion will arise.
For sequences $(\varphi_j)_j$ (resp. test functions $\varphi$), we define
\begin{eqnarray*}
	\mathcal{E}_n(\varphi) = \frac{1}{\sqrt{n}} \sum_j \varphi_j^2,\quad \mathcal{E}(\varphi) = \int \varphi^2(x)\, dx.
\end{eqnarray*} 
We denote the law of the stationary process $u^n$ by $\p_n$ and expected value with respect to $\p_n$ by $\esp_n$. 
As such, $u^n$ will be simply denoted by $u$, which can be seen as the canonical process under the law $\p_n$.
Note that, in this context, $\beta^{-2}\theta = \sqrt{n}+1/2$. We denote by $m_n$ the law of $-\log X+\tfrac12 \log n$ where $X\sim\text{Gamma}(\sqrt{n}+1/2)$ and $\mu_n = m_n^{\otimes \mathbb{N}}$ which, according to the previous discussion, turns out to be the stationary measure for $u$. 
As usual, $C$ denotes a positive constant whose value may change from line to line.


\section{Energy solutions of the Burgers equation}\label{sec:energy}
We will present the basics of the theory of  energy solutions of the stochastic Burgers equation as it was introduced in \cite{GJ-arma} and further developed in \cite{GuJ, GP-uniqueness} (see also \cite{GJS,GP-noneq, GP-generator}). Recall we are concerned with the equation
\begin{eqnarray}\label{eq:Burgers-2}
	\partial_t u = \frac12 \partial^2_x u + c \partial_x u -\frac{1}{2} \partial_x u^2 + \partial_x{\wh},
\end{eqnarray} 
where $\wh$ is a space-time white noise, i.e.~a distribution-valued centered Gaussian process with covariance $\esp[\wh(t,x)\wh(s,y)]=\delta(t-s)\delta(y-x)$. More precisely, $\wh$ acts on $L^2(\re_+ \times \re)$ in such a way that the random variables $\{\wh(f):\,f\in L^2(\re_+ \times \re)\}$ are jointly Gaussian with covariance
\begin{eqnarray*}
	\esp\left[ \wh(f_1)\wh(f_2)\right] = \int_{\re_+\times \re} f_1(t,x)f_2(t,x)\, dxdt,\quad f_1,f_2 \in L^2(\re_+ \times \re).
\end{eqnarray*}
 Due to the singularity of the noise, solutions to \eqref{eq:Burgers-2} can only be expected to be distribution-valued in space. The main difficulty then consists in giving a consistent meaning to the term $\partial_x u^2$. As we will see below, it is possible to make sense to this expression as a space-time distribution.

We start with a definition:
\begin{definition}
	We say that a process $\{u_t:\, t\in[0,T]\}$ satisfies condition (S) if, for all $t\in[0,T]$, the $\mathcal{S}'(\re)$-valued random variable $u_t$ is a white noise of variance $1$.
\end{definition}

For a process $\{u_t:\, t\in[0,T]\}$ satisfying condition (S), $0\leq s<t\leq T$, $\varphi\in\mathcal{S}(\re)$  and $\varepsilon>0$, we define
\begin{eqnarray*}
	\mathcal{A}^{\varepsilon}_{s,t}(\varphi) = \int^t_s \int_{\re} u_r(i_{\varepsilon}(x))^2 \partial_x\varphi(x)dxdr
\end{eqnarray*}
where $i_{\varepsilon}(x)=\varepsilon^{-1}{\bf 1}_{(x,x+\varepsilon]}$.

\begin{definition}
	Let $\{u_t:\, t\in[0,T]\}$ be a process satisfying condition (S). We say that $\{u_t:\, t\in[0,T]\}$ satisfies the energy estimate if there exists a constant $\kappa>0$ such that:
	
	\vspace{1ex}
	
	\noindent (EC1) For any $\varphi\in\mathcal{S}(\re)$ and any $0\leq s<t\leq T$,
							\begin{eqnarray*}
								\esp\left[ \left| \int^t_s u_r(\partial_x^2\varphi)\,dr\right|^2\right] \leq \kappa (t-s)\cale(\partial_x \varphi),
							\end{eqnarray*}

	\noindent (EC2) For any $\varphi\in\mathcal{S}(\re)$, any $0\leq s<t\leq T$ and any $0<\delta<\varepsilon <1$,
							\begin{eqnarray*}
								\esp\left[ \left| \mathcal{A}^{\varepsilon}_{s,t}(\varphi)-\mathcal{A}^{\delta}_{s,t}(\varphi) \right|^2\right] \leq \kappa (t-s)\varepsilon\cale(\partial_x \varphi).
							\end{eqnarray*}
\end{definition}
We state a key theorem proved in \cite{GJ-arma} which allows to give a sense to the quadratic term in \eqref{eq:Burgers-2}:
\begin{theorem}\label{thm:S-EC-imply}
	Assume $\{u_t:\, t\in[0,T]\}$ satisfies (S) and (EC2).
	Then, there exists an $\mathcal{S}'(\re)$-valued stochastic process $\{\mathcal{A}_t:\, t\in[0,T]\}$ with continuous paths such that
	\begin{eqnarray*}
		\mathcal{A}_t(\varphi) = \lim_{\varepsilon\to0}\mathcal{A}^{\varepsilon}_{0,t}(\varphi),
	\end{eqnarray*}
	in $L^2$, for any $t\in[0,T]$ and $\varphi\in\mathcal{S}(\re)$.
\end{theorem}
We are now ready to formulate the definition of an energy solution:

\begin{definition}\label{def:ES}
	We say that $\{u_t:\, t\in[0,T]\}$ is a stationary energy solution of the stochastic Burgers equation \eqref{eq:Burgers-2} if
	\begin{enumerate}[label=\arabic*.-]
		\item $\{u_t:\, t\in[0,T]\}$ satisfies (S), (EC1) and (EC2).
		
		\item For all $\varphi\in\mathcal{S}(\re)$, the process
				\begin{eqnarray*}
					u_t(\varphi)-u_0(\varphi)-\tfrac12 \int^t_0 u_s(\partial_x^2 \varphi)\,ds
					- c \int^t_0 u_s(\partial_x \varphi)\, ds					
					- \mathcal{A}_t(\varphi)
				\end{eqnarray*}
				is a martingale with quadratic variation $t\cale(\partial_x\varphi)$, where $\mathcal{A}$ is the process from Theorem \ref{thm:S-EC-imply}.
				
		\item For all $\varphi\in\mathcal{S}(\re)$, the process
				\begin{eqnarray*}
					u_{T-t}(\varphi)-u_T(\varphi)-\tfrac12 \int^t_0 u_{T-s}(\partial_x^2 \varphi)\,ds
					+ c \int^t_0 u_{T-s}(\partial_x \varphi)\, ds					
					+ \mathcal{A}_T(\varphi)
					- \mathcal{A}_{T-t}(\varphi)
				\end{eqnarray*}
				is a martingale with quadratic variation $t\cale(\partial_x\varphi)$.
	\end{enumerate}
\end{definition}
Existence of energy solutions was proved in \cite{GJ-arma}. Uniqueness was proved in \cite{GP-uniqueness}.


\section{System of SDEs and the Martingale Decomposition}

An application of It\^o's formula shows that, under $\p_n$, the collection $\{u_j:\, j\geq 1\}$ satisfies the system of SDEs:
\begin{eqnarray*}
	du_j &=& \left( W_{j-1} - W_j\right)\, dt + \beta \left( dB^{(j)}_t - dB^{(j-1)}_t\right), \quad j\geq 2,\\
	du_1 &=& (-\frac{\beta^2}{2}-W_1)\, dt + \beta \left( dB^{(1)}_t - dB^{(0)}_t\right),
\end{eqnarray*}
where $W_j = 1-e^{-u_j}$ and $\beta=n^{-1/4}$. As it will be noticed later, $\esp_n[W_j]=-\beta^2/2$. Writing $\bW_j=W_j-\esp_n[W_j]$ and setting $\bW_0=0$, the system above can be summarized as
\begin{eqnarray*}
	du_j &=& \left( \bW_{j-1} - \bW_j\right)\, dt + \beta \left( dB^{(j)}_t - dB^{(j-1)}_t\right), \quad j\geq 1.
\end{eqnarray*}
The initial condition is taken as
\begin{eqnarray*}
	u_j(0)=-\log X_j + \frac{1}{2}\log n,\quad j\geq 1,
\end{eqnarray*}
where $(X_j)_j$ is an i.i.d. family of $\text{Gamma}(\sqrt{n}+1/2)$ random variables.
Hence, the generator of this dynamics acts on smooth cylindrical functions as
\begin{eqnarray*}
	L &=& \frac{\beta^2}{2} \sum_j (\partial_j-\partial_{j-1})^2 + \sum_j (\bW_{j-1}-\bW_j) \partial_j
					\\
		&=&\frac{\beta^2}{2} \sum_j (\partial_j-\partial_{j-1})^2 + \sum_j \bW_j (\partial_{j+1}-\partial_{j}),
\end{eqnarray*}
where $\partial_j = \partial/\partial_{u_j}$. 
Remembering the definition of the density field
\begin{eqnarray*}
	\fieldn_t(\varphi) = \sum_j (u_j(tn)-\rho_n) \varphi^n_j, 
\end{eqnarray*}
Dynkin's formula implies that
\begin{eqnarray*}
	\mathcal{M}^n_t(\varphi)
	=
	\fieldn_t(\varphi)-\fieldn_0(\varphi) - \int^t_0 \left( \partial_s + nL \right) \fieldn_s(\varphi)\, ds
	=
	\beta \int^t_0\sum_j \gradn \varphi^n_j dB^{(j)}_s
\end{eqnarray*}
is a martingale with quadratic variation
\begin{eqnarray*}
	\langle \mathcal{M}^n(\varphi)\rangle_t 
	=
	\frac{\beta^2 }{2} \int^t_0\sum_j (\gradn \varphi^n_j)^2ds.
\end{eqnarray*}
Note that the time integral cannot be removed as the discretization of $\varphi$ depends on time. 
By integration-by-parts, we can formally obtain $L^*$, the adjoint of $L$ in $L^2(\mu_n)$:
\begin{eqnarray*}
	L^*= \frac{\beta^2}{2} \sum_j (\partial_j-\partial_{j-1})^2 - \sum_j \bW_j (\partial_{j+1}-\partial_{j}).
\end{eqnarray*}
This allows us to identify the symmetric and anti-symmetric parts of the generator:
\begin{eqnarray*}
	&&S = \frac{L+L^*}{2} 
	= \sum_j\left\{ \frac{\beta^2}{2}(\partial_{j+1}-\partial_j)^2-\frac12 \bW_j\,(\partial_{j+1}+\partial_{j-1}-2\partial_j)\right\} 
	\\
	&&A = \frac{L-L^*}{2}
	= \frac12 \sum_j \bW_j\, (\partial_{j+1}-\partial_{j-1}).
\end{eqnarray*}
With this at hands, we can properly decompose the dynamics: remembering $\beta=n^{-1/4}$, the symmetric part corresponds to
\begin{eqnarray*}	
	\spart^n_t(\varphi)
	=
	\int^t_0 n S \fieldn_s(\varphi)\, ds
	=
	\frac12 \int^t_0 \sum_j \bW_j(sn) \lapln \varphi^n_j ds
\end{eqnarray*}
while the anti-symmetric part corresponds to
\begin{eqnarray*}
	\apart^n_t(\varphi)
	&=&
	\int^t_0 \left( 
		\partial_s + nA
	\right)
	\fieldn_s(\varphi)\, ds\\
	&=&
	\int^t_0 \sqrt{n}\sum_j \left\{ \bW_j(sn)\gradn \varphi^n_j - (u_j(sn)-\rho_n)\partial_x \varphi^n_j\right\}ds.
\end{eqnarray*}

\section{Static Estimates}
We briefly recall some facts about the Gamma and log-Gamma distributions.
If $X\sim \text{Gamma}(\nu)$, then
\begin{eqnarray*}
	\p\left[ X \geq x \right] &=& \frac{1}{\Gamma(\nu)} \int^{\infty}_x y^{\nu-1}e^{-y}dy,
\end{eqnarray*}
where $\Gamma(\nu)=\int^{\infty}_0 y^{\nu-1}e^{-y}dy$ is the Gamma function. By explicit computations,
\begin{eqnarray*}
	\esp[X] = \nu \quad \text{and} \quad \Vvv[X] = \nu.
\end{eqnarray*}
Now, if we take $\beta=n^{-1/4}$, $\theta=1+1/(2\sqrt{n})$ and let $\nu= \beta^{-2}\theta=\sqrt{n}+1/2$, we obtain
\begin{eqnarray*}
	\esp_n[W] &=& \esp_n[1-e^{-u}] = -\frac{1}{2\sqrt{n}} = -\frac{\beta^2}{2},\\
	\Vvv_n[W] &=& \Vvv_n[e^{-u}] = \frac{1}{\sqrt{n}} + \frac{1}{2n}=\beta^2\theta,
\end{eqnarray*}
as, under $\p_n$, $e^{-u}\sim \beta^2X$ with $X\sim \text{Gamma}(\sqrt{n}+1/2)$. Here, $\Vvv_n$ denotes the variance with respect to $\p_n$. On the other hand, for $X\sim \text{Gamma}(\sqrt{n}+1/2)$,
\begin{eqnarray*}
	\p\left[ -\log X -\log \beta^2 \geq x \right]
	=
	\frac{1}{\beta^{2\nu}\Gamma(\nu)} \int^{\infty}_x e^{-\nu y - \beta^{-2}e^{-y}}dy,
\end{eqnarray*}
from where we can compute
\begin{eqnarray*}
	\esp_n\left[ u \right] = -\Psi_0(\sqrt{n}+\tfrac12) \quad \text{and} \quad
	\Vvv_n\left[ u \right] = \Psi_1(\sqrt{n}+\tfrac12),
\end{eqnarray*}
with $\Psi_0=\Gamma'/\Gamma$ and $\Psi_1=\Psi_0'$. Asymptotics of these functions are known \cite{AS}:
\begin{eqnarray*}
	\Psi_0(x) = \log x - \frac{1}{2x} + O\left(\frac{1}{x^2}\right),\quad \Psi_1(x) = O\left(\frac{1}{x}\right),\quad \text{as} \quad x\to\infty.
\end{eqnarray*}
From this, we conclude that
\begin{eqnarray*}
	\esp_n[u]=-\frac{1}{2\sqrt{n}}+O\left(\frac{1}{n}\right) \quad \text{and} \quad \Vvv_n[u] = O\left(\frac{1}{\sqrt{n}}\right).
\end{eqnarray*}
The following lemma provides bounds for higher moments:
\begin{lemma}
	Let $F$ be a locally bounded function such that $|F(x)|e^{-c|x|}$ is bounded for some constant $c>0$ and such that there exists $C>0$, $a>1$ and $k\geq 1$ such that
	\begin{eqnarray*}
		|F(x)| \leq C |x|^k,\quad \forall \, x \in [-a,a].
	\end{eqnarray*}
	Then, there exists $C'>0$ such that
	\begin{eqnarray*}
		\esp_n[|F(u)|] \leq C' n^{-k/4}.
	\end{eqnarray*}
\end{lemma}
\begin{proof}
	Write again $\beta=n^{-1/4}$, $\theta=1+1/(2\sqrt{n})$ and $\nu= \beta^{-2}\theta=\sqrt{n}+1/2$.
	First, an application of Stirling's formula shows that
	\begin{eqnarray*}
		\beta^{2\nu} \Gamma(\nu) \geq c e^{-\nu}\beta,
	\end{eqnarray*}
	for some $c>0$. Next, allowing the value of $C$ to change from line to line,
	\begin{eqnarray*}
		e^{\nu}\int^{a}_{-a} |F(y)|\,e^{-\nu y-\beta^{-2}e^{-y}}dy
		&\leq&
		C\int^{a}_{-a} |y|^k\,e^{\beta^{-2}(1- y-e^{-y})}dy\\
		&\leq&
		C\int^{a}_{-a} |y|^k\,e^{-c_1\beta^{-2}y^2}dy\\
		&\leq&
		C\beta^{k+1}\int^{\beta^{-1}a}_{-\beta^{-1}a} |y|^k\,e^{-c_1y^2}dy
		\leq  C\beta^{k+1},
	\end{eqnarray*}
	for some $c_1>0$.
	We bound the contributions of $[-a,a]^c$: for some small enough $c_2>0$, we have
	\begin{eqnarray*}
		e^{\nu}\int^{\infty}_{a} |F(y)|\,e^{-\nu y-\beta^{-2}e^{-y}}dy
		&\leq&
		C\int^{\infty}_{a} e^{-c_2\nu (y-1)}dy
		\,\,\, \leq \,\,\, C \nu^{-1}e^{-c_2 \nu (a-1)},
	\end{eqnarray*}
	and
	\begin{eqnarray*}
		e^{\nu}\int^{-a}_{-\infty} |F(y)|\,e^{-\nu y-\beta^{-2}e^{-y}}dy 
		&\leq&
		\int^{-a}_{-\infty} e^{-c_3\beta^{-2}e^{-y}}dy \,\,\, \leq \,\,\, C \nu^{-1}e^{-c_3\nu e^{a}},
	\end{eqnarray*}
	for some $c_3>0$.
\end{proof}
In particular, for each $k\geq 1$, we can find constants $C_k>0$ such that
\begin{eqnarray*}
	\esp_n[|u|^k] \leq C_k n^{-k/4}\quad \text{and} \quad \esp_n[|W|^k] \leq C_k n^{-k/4}.
\end{eqnarray*}


\section{Dynamical estimates}
We denote by $\mathscr{C}$ the collection of cylindrical functions $F$ of the form $F(u)=f(u_{-n},\cdots, u_n)$ for some $n\geq 0$ and some $f\in C^2(\re^{2n+1})$ with polynomial growth of its derivatives up to order $2$.
We recall the Kipnis-Varadhan estimate:
\begin{eqnarray*}
	\esp_n\left[ \sup_{t\leq T}\left| \int^t_0 F(s,u(sn))\, ds\right|^2\right]
	\leq C \int^T_0 \|F(s,\cdot)\|_{-1,n}ds,
\end{eqnarray*}
where the $||\cdot||_{-1,n}$-norm is defined through the variational formula
\begin{eqnarray*}
	\|F\|_{-1,n} = \sup_{f\in \mathscr{C}}\left\{ 2 \int F(u)fd\mu_n+n \int f Lf d\mu_n\right\}
\end{eqnarray*}
The proof is a straightforward adaptation of \cite{DGP}, Corollary 3.5.
Note that
\begin{eqnarray*}
	-\int f Lf d\mu_n = \frac{\beta^2}{2}\sum_j \int \left( (\partial_{j+1}-\partial_j)f\right)^2d\mu_n
\end{eqnarray*}
so that
\begin{eqnarray*}
	\|F\|_{-1,n}^2 = \sup_{f\in \mathscr{C}}\left\{ 2 \int F(u)fd\mu-\frac{\sqrt{n}}{2}\sum_j \int \left( (\partial_{j+1}-\partial_j)f\right)^2d\mu\right\}.
\end{eqnarray*}
Next, we notice that our model satisfies the integration-by-parts formula:
\begin{eqnarray*}
	\int (\bW_{j+1}-\bW_j)fd\mu_n = \beta^2\int (\partial_{j+1}-\partial_j) f \, d\mu_n.
\end{eqnarray*}

\subsection{One-block estimate}

Recall $\bW_j = W_j - \esp_n[W_j]$ and let $\overrightarrow{W}^l_j = \displaystyle \frac{1}{l}\sum^{j+l-1}_{k=j}\bW_k$ for $l\geq 2$. Let also $\tau_j$ denote the canonical shift: $\tau_j u_i=u_{i+j}$. In the following, we consider test functions $(\varphi_j)_j$ which may depend on time.

\begin{lemma}
	Let $l\geq 2$ and let $g$ be a function with zero-mean respect to $\mu_n$ such that $g(u)=\tilde g(u_{j_0})$ for some $\tilde g:\re \to \re$ and $ j_0 \notin \{0,\cdots,l-1\}$. Write $g_j = g(\tau_j u)$. 
	There exists a constant $C>0$ such that
	\begin{eqnarray*}
	\esp_n\left[ \sup_{t\leq T}\left| \int^t_0 \sqrt{n}\sum_j \varphi_j g_{j}(\bW_{j}(sn)-\overrightarrow{W}^l_j(sn))\,ds\right|^2\right]
	\leq C l \| g \|_{L^2(\mu_n)}^2 \int^T_0 \mathcal{E}_n(\varphi_t)\,dt,
	\end{eqnarray*}
	where $\mathcal{E}_n(\varphi) = \displaystyle \frac{1}{\sqrt{n}} \sum_j \varphi_j^2$.
\end{lemma}
\begin{proof}
	First, we observe that
	\begin{eqnarray*}
		\sqrt{n}\sum_j \varphi_j g_{j}(\bW_j-\overrightarrow{W}^l_j)
		=
		\sqrt{n}\sum_j \varphi_j g_{j}\sum^{l-2}_{i=0}(\bW_{j+i}-\bW_{j+i+1})\psi_i,
	\end{eqnarray*}
	for  $\psi_i=(l-i)/l$. Rearranging the sum (simply put $k=j+i$), 
	\begin{eqnarray*}
		\sqrt{n}\sum_j \varphi_j g_{j}(\bW_j-\overrightarrow{W}^l_j)
		&=&
		\sqrt{n} \sum_k \left( \sum^{l-2}_{i=0} \varphi_{k-i}g_{k-i}\psi_i\right) (\bW_k-\bW_{k+1})\\
		&=&
		\sqrt{n} \sum_k F_k (\bW_k-\bW_{k+1}),
	\end{eqnarray*}
	where $F_k:=\sum^{l-2}_{i=0} \varphi_{k-i}g_{k-i} \psi_i$.
	Hence, for $f\in \mathscr{C}$, using integration-by-parts and our hypothesis on $g$,
	\begin{eqnarray*}
		&&2\int\sqrt{n}\sum_j \varphi_j g_{j}(\bW_j-\overrightarrow{W}^l_j) f(u) d\mu_n\\
		&=&
		2\int\sqrt{n} \sum_k F_k (\bW_k-\bW_{k+1})f(u) d\mu_n\\
		&=&
		2\beta^2 \sqrt{n} \sum_k F_k  (\partial_k-\partial_{k+1})f(u) d\mu_n\\
		&\leq&
		\int \sum_k \left\{ \alpha F^2_k + \frac{1}{\alpha}((\partial_k-\partial_{k+1})f(u))^2\right\}d\mu_n,
	\end{eqnarray*}
	by Young's inequality and $\beta^2 \sqrt{n}=1$. Taking $\alpha = 2 n^{-1/2}$, we get that the above is bounded by
	\begin{eqnarray*}
		\frac{2}{\sqrt{n}} \sum_k \int F^2_k d\mu + \frac{\sqrt{n}}{2}\sum_k \int ((\partial_k-\partial_{k+1})f(u))^2 d\mu_n
	\end{eqnarray*}
	which yields the bound
	\begin{eqnarray*}
		\left\| \sqrt{n}\sum_j \varphi_j g_{j}(\bW_j-\overrightarrow{W}^l_j) \right\|_{-1,n}^2 
		\leq 
		\frac{2}{\sqrt{n}} \sum_k \int F^2_k d\mu_n.
	\end{eqnarray*}
	Finally, 
	\begin{eqnarray*}
		\sum_k\int F^2_k d\mu_n &=& \sum_k\sum^{l-2}_{i=0} \varphi^2_{k-i} \int g^2 d\mu_n 
		\leq Cl \sqrt{n}\int g^2 d\mu_n \, \mathcal{E}_n(\varphi).
	\end{eqnarray*}
	The result follows from the Kipnis-Varhadan estimate.
\end{proof}

\subsection{The second-order Boltzmann-Gibbs principle}
Let 
\begin{eqnarray*}
	Q(l,t) = \left( \overrightarrow{W}^l_0(t)\right)^2 - \frac{\sigma_n^2}{l},
	\quad \text{with} \quad
	\sigma_n^2:= \beta^2 + \frac{\beta^4}{2}=\Vvv_n[W].	
\end{eqnarray*}
The following is the central estimate in our proof:
\begin{proposition}\label{thm:second-order}
	\begin{eqnarray*}
		\esp_n\left[\sup_{t\leq T}\left| \int^t_0 \sqrt{n}\sum_j \left\{ 
				\bW_{j-1}(sn)\bW_{j}(sn)-
				\tau_jQ(l,sn)
		\right\}\varphi_j ds\right|^2\right]
		&\leq&
		C
		\left( \frac{l}{\sqrt{n}} + \frac{T}{l^2}\right)
		\int^T_0\mathcal{E}_n(\varphi_t)\,dt.
	\end{eqnarray*}
\end{proposition}
\begin{proof}
	Decompose as follows:
	\begin{eqnarray*}
		\bW_{j-1}\bW_{j}-
				\tau_jQ(l)
		&=&
		\bW_{j-1}[\bW_{j}-\overrightarrow{W}^l_j]\\
		&+&
		\overrightarrow{W}^l_j[\bW_{j-1}-\overrightarrow{W}^l_j] + \frac{\beta^2}{l}\overrightarrow{W}^l_j + \frac{\sigma_n^2}{l}\\
		&-&
		\frac{\beta^2}{l}\overrightarrow{W}^l_j .
	\end{eqnarray*}
	The first term is handled with the one-block estimate with $g=\bW_{-1}$ together with $\esp_n[|\bW|^2]\leq C n^{-1/2}$ and gives the bound with the $l/\sqrt{n}$ term. The second one is the object of the next lemma and gives the same bound. The third one can be estimated by a careful $L^2$ computation and gives the bound with the $T/l^2$ term: using $\beta^2 \sqrt{n}=1$, applying Jensen's inequality, Tonelli and stationarity,
	\begin{eqnarray*}
		\esp_n\left[ \sup_{t\leq T}\left| \int^t_0 \sqrt{n} \frac{\beta^2}{l} \sum_j \overrightarrow{W}^l_j(sn) \varphi_j ds\right|^2\right]
		&\leq&
		\frac{1}{l^2} 
		\esp_n\left[ \sup_{t\leq T} t \int^t_0 \left|\sum_j \overrightarrow{W}^l_j(sn) \varphi_j \right|^2 ds\right]
		\\
		&\leq&
		\frac{T}{l^2} 
		\int^T_0\esp_n\left[ \left|\sum_j \overrightarrow{W}^l_j(sn) \varphi_j \right|^2 \right]  ds.
	\end{eqnarray*}
	Next, we have to take dependencies into account to compute the expected value: using again Jensen's inequality and the independence of $\overrightarrow{W}^l_j$ and $\overrightarrow{W}^l_k$ if $|j-k|\geq l$,
	\begin{eqnarray*}
		\esp_n\left[ \left|\sum_j \overrightarrow{W}^l_j\varphi_j \right|^2 \right]
		&=&
		\esp_n\left[ \left|\sum^{l-1}_{k=0}\sum_j \overrightarrow{W}^l_{lj+k}\varphi_{lj+k} \right|^2 \right]\\
		&\leq&
		l \sum^{l-1}_{k=0}\esp_n \left[ \left|\sum_j \overrightarrow{W}^l_{lj+k}\varphi_{lj+k} \right|^2 \right]\\
		&=&
		l \sum^{l-1}_{k=0} \sum_j \esp_n\left[ \left|\overrightarrow{W}^l_{lj+k}\right|^2 \right]\varphi_{lj+k}^2 \\
		&\leq&
		\frac{C}{\sqrt{n}}
		\sum^{l-1}_{k=0} \sum_j \varphi_{lj+k}^2
		=
		C \mathcal{E}_n(\varphi),
	\end{eqnarray*}
	as $\esp_n[|\overrightarrow{W}^l|^2]\leq C/l\sqrt{n}$.
\end{proof}
The following lemma finishes the proof of the Boltzmann-Gibbs principle:
\begin{lemma}
	\begin{eqnarray*}
		&&\esp_n\left[\sup_{t\leq T}\left| \int^t_0 \sqrt{n}\sum_j \left\{ 
				\overrightarrow{W}^l_j(sn) [\bW_{j-1}(sn)-\overrightarrow{W}^l_j(sn)] 
				+ \frac{\beta^2}{l}	\overrightarrow{W}^l_j(sn)
				+\frac{\sigma_n^2}{l}
		\right\}\varphi_j ds\right|^2\right]\\
		&&
		\hspace{65ex}
		\leq
		C\frac{l}{\sqrt{n}} \int^T_0\mathcal{E}_n(\varphi_t)\, dt.
	\end{eqnarray*}
\end{lemma}
\begin{proof}
	Let $f\in\mathscr{C}$.
	We begin with a computation: 
	\begin{eqnarray*}
		\int \sqrt{n}\sum_j \varphi_j \overrightarrow{W}^l_j [\bW_{j-1}-\overrightarrow{W}^l_j]f\, d\mu_n
		&=&
		\int \sqrt{n}\sum_j \varphi_j \overrightarrow{W}^l_j
		\sum^{l-1}_{k=0} \psi_k (\bW_{j+k-1}-\bW_{j+k})		
		fd\mu_n,
	\end{eqnarray*}
	where $\psi_k = (l-k)/l$. We will apply integration-by-parts: for $k\geq 1$,
	\begin{eqnarray*}
		&&\int \overrightarrow{W}^l_j (\bW_{j+k-1}-\bW_{j+k})fd\mu_n \\
		&&= 
		\beta^2\int \overrightarrow{W}^l_j (\partial_{j+k-1}-\partial_{j+k}) f d\mu_n
		+ \frac{\beta^2}{l}\int (\bW_{j+k}-\bW_{j+k-1})\, f d\mu_n.		
	\end{eqnarray*}
	The term $k=0$ has to be handled separately:
	\begin{eqnarray*}
		&&\int \overrightarrow{W}^l_j (\bW_{j-1}-\bW_{j})fd\mu_n  \\
		&&= 
		\beta^2\int \overrightarrow{W}^l_j (\partial_{j-1}-\partial_{j}) f d\mu_n + \frac{\beta^2}{l}\int \bW_j f d\mu_n
		-\frac{\sigma_n^2}{l}\int f\, d\mu_n.
	\end{eqnarray*}
	Carefully recombining the terms yields the identity
	\begin{eqnarray*}
		&&\int \sum_j \varphi_j \left\{\overrightarrow{W}^l_j [\bW_{j-1}-\overrightarrow{W}^l_j]
			+ \frac{\beta^2}{l} \overrightarrow{W}^l_j + \frac{\sigma_n}{l}
		\right\}f\, d\mu_n \\
		&& \quad = 
		\beta^2\int \sum_j \varphi_j \overrightarrow{W}^l_j(sn) 
		\sum^{l-1}_{k=0}\psi_k(\partial_{j+k-1}-\partial_{j+k})fd\mu_n.
	\end{eqnarray*}
	By Young's inequality, twice the above is bounded by
	\begin{eqnarray*}
		\beta^2\int \sqrt{n}\sum_j\sum^{l-1}_{k=0}\psi_k
		\left\{
			\alpha \varphi^2_j (\overrightarrow{W}^l_j)^2 + \frac{1}{\alpha} ((\partial_{j+k-1}-\partial_{j+k})f)^2
		\right\}d\mu_n.
	\end{eqnarray*}
	Taking $\alpha=2l/\sqrt{n}$ and using $\beta^2 \sqrt{n}=1$, the bound becomes
	\begin{eqnarray*}
		&&\int \sum_j\sum^{l-1}_{k=0}\psi_k
		\left\{
			\frac{2l}{\sqrt{n}}\varphi^2_j (\overrightarrow{W}^l_j)^2 + \frac{\sqrt{n}}{2l} ((\partial_{j+k-1}-\partial_{j+k})f)^2
		\right\}d\mu_n\\
		&& \quad \leq
		C\frac{l^2}{\sqrt{n}} \sum_j \varphi_j^2 \esp_n[(\overrightarrow{W}^l_j)^2] + \frac{\sqrt{n}}{2} \int \sum_j((\partial_{j-1}-\partial_{j})f)^2d\mu_n.
	\end{eqnarray*}
	The result follows from Kipnis-Varadhan inequality and the bound $\displaystyle \esp_n[(\overrightarrow{W}^l_j)^2] \leq C \frac{1}{l\sqrt{n}}$.
\end{proof}

\section{Tightness}

We will use Mitoma's criterion \cite{Mitoma}: a sequence $(\mathcal{Y}^n)_n$ is tight in $C([0,T],\mathcal{S}'(\re))$ if and only if $(\mathcal{Y}^n(\varphi))_n$ is tight in $C([0,T],\re)$ for all $\varphi\in \mathcal{S}(\re)$.

\subsection{Martingale term}
Recall that
\begin{eqnarray*}
	\mathcal{M}^n_t(\varphi)
	&=&
	\beta \int^t_0\sum_j \gradn \varphi^n_j dB^{(j)}_s
\end{eqnarray*}
has quadratic variation 
\begin{eqnarray*}
	\langle \mathcal{M}^n(\varphi)\rangle_t = \displaystyle \frac{\beta^2}{2}\int^t_0\sum_j (\gradn \varphi^n_j)^2ds
	\leq
	Ct \mathcal{E}(\partial_x\varphi).
\end{eqnarray*}
Hence, from the Burkholder-Davis-Gundy inequality, it follows that
\begin{eqnarray*}
	\esp_n\left[ \left| \mathcal{M}^n_{t_2}(\varphi)-\mathcal{M}^n_{t_1}(\varphi)\right|^p\right]
	\leq C |t_2-t_1|^{p/2} \mathcal{E}(\partial_x \varphi)^{p/2},
\end{eqnarray*}
for all $p\geq 1$. Tightness follows from Kolmogorov's criterion by taking $p$ large enough.
\subsection{Symmetric term}
Recall that
\begin{eqnarray*}
	\spart^n_t(\varphi)
	=
	\frac12 \int^t_0 \sum_j \bW_j(sn)\lapln \varphi^n_j ds.
\end{eqnarray*}
Tightness follows at once from an $L^2$ bound:
\begin{eqnarray*}
	\esp_n\left[ \left|
		\spart^n_{t_2}(\varphi)-\spart^n_{t_1}(\varphi)
	\right|^2
	\right]
	&\leq&
	C |t_2-t_1| \int^{t_2}_{t_1} \sum_j \esp_n \left[ \bW(sn)^2\right] (\lapln \varphi^n_j)^2ds\\
	&\leq&
	C |t_2-t_1|^2\mathcal{E}(\partial_x^2\varphi),
\end{eqnarray*}
where we used $\esp_n \left[ \bW(sn)^2\right]=O(n^{-1/2})$.
\subsection{Anti-symmetric term}
Recall
\begin{eqnarray*}
	\apart^n_t(\varphi)
	&=&
	\int^t_0 \sqrt{n}\sum_j \left\{ W_j(sn)\gradn \varphi^n_j - (u_j(sn)-\rho_n)\partial_x \varphi^n_j\right\}ds\\
	&=&
	\int^t_0 \sqrt{n}\sum_j \left( W_j(sn)-u_j(sn)\right)\gradn \varphi^n_j ds
	+ E^n_t(\varphi)
\end{eqnarray*}
where
\begin{eqnarray*}
	E^n_t(\varphi) = \int^t_0 \sqrt{n} \sum_j \left(u_j(sn)-\rho_n\right)\left(\gradn \varphi^n_j-\partial_x \varphi^n_j\right)\,ds.
\end{eqnarray*}
Using $\esp_n[|u_j(sn)-\rho_n|^2]=O(n^{-1/2})$ and the mean-value theorem,
\begin{eqnarray*}
	\esp_n\left[ \sup_{t\leq T}\left|E^n_t(\varphi) \right|^2\right]
	\leq
	n T \int^T_0 \sum_j
	\esp_n\left[ \left| u_j(sn)-\rho_n\right|^2 \right]
	\left|\gradn \varphi^n_j-\partial_x \varphi^n_j \right|^2
	\,ds \leq C \frac{T^2}{n}. 
\end{eqnarray*}
As a consequence, $E^n_{\cdot}(\varphi)$ converges to $0$ in the ucp topology.

Now, a naive Taylor expansion suggests that
\begin{eqnarray*}
	\sqrt{n}\left( W_j(sn)-u_j(sn)\right) = - \frac12\sqrt{n} u_j(sn)^2 + O(\sqrt{n}u_j(sn)^3),
\end{eqnarray*}
explaining in particular the emergence of the quadratic term. This simple argument has two flaws: first, we are unable to handle the quadratic term as is, and second, the order three terms cannot be neglected based on moments considerations only. However, order four and higher terms can be neglected: 
\begin{eqnarray*}
	\esp_n\left[
		\sup_{t\leq T}
		\left|
			\sqrt{n} \int^t_0 \sum_j u_j(sn)^k \gradn\varphi^n_j ds
		\right|^2
	\right]
	&\leq&
	nT \int^T_0 \sum_j \esp_n\left[ u_j(sn)^{2k} \right] (\gradn\varphi^n_j)^2 ds\\
	&\leq&
	C T^2 n^{\frac{3-k}{2}}.
\end{eqnarray*}
A similar bound holds for powers of $\bW$.
We proceed now to a Taylor expansion which will be more useful to us: first,
\begin{eqnarray*}
	W_j-u_j = -\frac12 u_j^2+\frac16 u_j^3 + O(u_j^4 + |W_j|^5).
\end{eqnarray*}
Here, the error of order $u^4_j$ takes into consideration positive values of $u_j$ while the term $|W_j|^5$ is included to account for large negative values of $u_j$.
On the other hand,
\begin{eqnarray*}
	aW_ju_j+bW_ju_j^2 = au^2+(b-\frac{a}{2})u_j^3+O(u_j^4 + |W_j|^5).
\end{eqnarray*}
Equating both expansions and setting $a=-1/2$ and $b=-1/12$, we obtain
\begin{eqnarray*}
	W_j-u_j = -\frac12 W_ju_j - \frac{1}{12} W_ju_j^2+O(u_j^4+ |W_j|^5).
\end{eqnarray*}
Keeping in mind the nature of our dynamical estimates, we must find a way to `shift' the index of one of the terms in each product in the right-hand-side. We use the identities
\begin{eqnarray*}
	Lu_j^2 &=& 2(W_{j-1}-W_j)u_j + 2\beta^2,\\
	Lu_j^3 &=& 3(W_{j-1}-W_j)u_j^2+6\beta^2u_j,
\end{eqnarray*}
yielding
\begin{eqnarray*}
	W_j-u_j = -\frac12 W_{j-1}u_j - \frac{1}{12}W_{j-1}u_j^2+\frac14 Lu^2_j+\frac{1}{36}Lu_j^3-\frac{\beta^2}{2}-\frac{1}{6}\beta^2 u_j + O(u_j^4+ W_j^4).
\end{eqnarray*}
We use Taylor expansions one last time to switch between $W_{j-1}u_j$ and $W_{j-1}W_j$:
\begin{eqnarray*}
	W_{j-1}W_j = W_{j-1}u_j - \frac12 W_{j-1}u_j^2 + \frac16 W_{j-1}u_j^3+O(u_j^4+ W_j^4).
\end{eqnarray*}
Hence,
\begin{eqnarray*}
	W_j-u_j 
	&=& -\frac12 W_{j-1}W_j - \frac{1}{3}W_{j-1}u_j^2+\frac14 Lu^2_j+\frac{1}{36}Lu_j^3\\
	&&
	-\frac{\beta^2}{2}-\frac{1}{6}\beta^2 u_j + \frac{1}{12}W_{j-1}u_j^3+ O(u_j^4 + |W_j|^5).
\end{eqnarray*}
We will investigate the convergence of each of these terms separately. 
The first (and main) term will be treated at the end of the section.
The analysis of the second term is rather lengthy and will be left for the appendix.
The terms involving $L$ are treated in Lemma \ref{thm:lemma-Lu2} and \ref{thm:lemma-Lu3} below.
The term $\beta^2 u_j$ is easily seen to be tight. 
Finally, the term involving $W_{j-1}u^3_j$ can be neglected by means of an $L^2$ computation.

\begin{lemma}\label{thm:lemma-Lu2}
	There exists a constant $C>0$ such that
	\begin{eqnarray*}
		\esp_n\left[ \sup_{t\leq T} \left|\sqrt{n}\int^t_0  \sum_j L u_j^2(sn) \gradn\varphi_j^n ds\right|^2\right]
		\leq
		C \frac{T}{\sqrt{n}}.
	\end{eqnarray*}
\end{lemma}
\begin{proof}
	Let $g\in \mathscr{C}$. By integration-by-parts,
\begin{eqnarray*}
	\int  L u_j^2 g d\mu_n
	&=&
	2\int (\bW_{j-1}-\bW_j)u_j g d\mu_n +2\beta^2 \int g d\mu_n \\
	&=& 
	2\beta^2 \int u_j (\partial_{j-1}-\partial_j) gd\mu_n.
\end{eqnarray*}
Hence, by Young's inequality,
\begin{eqnarray*}
	&&2\int \sqrt{n}\sum_j L u_j^2 g \gradn\varphi_j^n d\mu_n
	\leq 4 \int \sum_j \left\{
		\alpha u_j^2 (\gradn\varphi_j^n)^2 + \frac{1}{\alpha} \left( (\partial_{j-1}-\partial_j) g\right)^2
	\right\} d\mu_n.
\end{eqnarray*}
With $\alpha=8/\sqrt{n}$, this is further bounded by
\begin{eqnarray*}
	&&\frac{16}{\sqrt{n}} \sum_j \esp_n[u_j^2] (\gradn\varphi_j^n)^2 + \frac{\sqrt{n}}{2} \int \sum_j \left( (\partial_{j-1}-\partial_j) g\right)^2 d\mu_n \\
	&&
	\leq \frac{C}{\sqrt{n}} + \frac{\sqrt{n}}{2} \int \sum_j \left( (\partial_{j-1}-\partial_j) g\right)^2 d\mu_n.
\end{eqnarray*}
The result follows from Kipnis-Varadhan inequality.
\end{proof}
\begin{lemma}\label{thm:lemma-Lu3}
	There exists a constant $C>0$ such that
	\begin{eqnarray*}
		\esp_n\left[ \sup_{t\leq T} \left|\sqrt{n}\int^t_0  \sum_j L u_j^3(sn) \gradn\varphi^n_j ds\right|^2\right]
		\leq
		C \frac{T}{n}.
	\end{eqnarray*}
\end{lemma}
\begin{proof}
	Let $g\in \mathscr{C}$. By integration-by-parts,
	\begin{eqnarray*}
	 \int Lu^3_j g d\mu_n
	&=&
	3\int (\bW_{j-1}-\bW_j)u^2_j g d\mu_n + 6 \beta^2 \int u_j g d\mu_n\\
	&=& 3\beta^2 \int u_j^2 (\partial_{j-1}-\partial_j) gd\mu_n.
\end{eqnarray*}
	The proof is then similar to the previous lemma.
\end{proof}
We now focus on the term $W_{j-1}W_j$. Note that
\begin{eqnarray*}
	W_{j-1}W_j = \bW_{j-1}\bW_j + \frac{1}{\sqrt{n}} [W_{j-1}+W_j] - \frac{1}{n}.
\end{eqnarray*}
An $L^2$ computation easily shows that the contribution of the linear terms is tight. The term $\frac{1}{n}$ will disappear as we only test against gradients. We are left to show the tightness of the term 
\begin{eqnarray*}
	\widetilde{\mathcal{B}}^n_t(\varphi)
	=
	\int^t_0 \sqrt{n} \sum_j \bW_{j-1}(sn)\bW_j(sn) \gradn \varphi^n_j ds.
\end{eqnarray*}
By Proposition \ref{thm:second-order} and stationarity,
\begin{eqnarray*}
	&&\esp_n\left[ \left|
		\widetilde{\mathcal{B}}^n_{t_2}(\varphi)-\widetilde{\mathcal{B}}^n_{t_1}(\varphi)
		-
		\int^{t_2}_{t_1} \sqrt{n} \sum_j
		\tau_jQ(l,sn) \gradn \varphi^n_j ds
	\right|^2\right]
	\leq C \left( \frac{(t_2-t_1)l}{\sqrt{n}} + \frac{(t_2-t_1)^2}{l^2}\right).
\end{eqnarray*}
On the other hand, a careful $L^2$ computation taking dependencies into account shows that
\begin{eqnarray*}
	\esp_n\left[ \left|
	\int^{t_2}_{t_1} \sqrt{n} \sum_j
		\tau_jQ(l,sn) \gradn \varphi^n_j ds
	\right|^2\right]
	\leq C \frac{(t_2-t_1)^2\sqrt{n}}{l}.
\end{eqnarray*}
For $\frac{1}{n} \leq t_2-t_1 \leq 1$, we can take $l\sim \sqrt{(t_2-t_1)n}$ in the above two inequalities to get
\begin{eqnarray*}
	\esp_n\left[ \left|
		\int^{t_2}_{t_1} \sqrt{n} \sum_j
			\bW_{j-1}(sn)\bW_{j}(sn) \gradn \varphi^n_j ds
	\right|^2\right]
	\leq
	C (t_2-t_1)^{3/2}.
\end{eqnarray*}
For $t_2-t_1 \leq \frac{1}{n}$, a crude $L^2$ bound yields
\begin{eqnarray*}
	\esp_n\left[ \left|
		\widetilde{\mathcal{B}}^n_{t_2}(\varphi)-\mathcal{\widetilde{B}}^n_{t_1}(\varphi)
	\right|^2\right]
	\leq
	C (t_2-t_1)^2 \sqrt{n}
	\leq
	C (t_2-t_1)^{3/2}.
\end{eqnarray*}
This proves tightness.

\section{Identification of the Limit}

By tightness, we obtain processes $\field,\, \spart,\, \mpart$ and $\mart$ such that
\begin{eqnarray*}
	\lim_{n\to\infty}\field^n &=& \field, \phantom{blal}\lim_{n\to\infty} \spart^n = \spart,\\
	\lim_{n\to\infty} \mpart^n &=& \mpart, \phantom{bla} \lim_{n\to\infty} \mart^n = \mart,
\end{eqnarray*}
along a subsequence that we still denote by $n$.
\subsection{Convergence at fixed times}
A straightforward  adaptation of the arguments in \cite{DGP}, Section 4.1.1, shows that $\field^n_t$ converges to a white noise for each fixed time $t\in[0,T]$. This in turns proves that the limit satisfies property (S).
\subsection{Linear terms }\label{sec:linear-terms}
We now consider the terms involving the expressions $\beta^2 u_j$ and $\displaystyle \frac{1}{\sqrt{n}} W_j$. 
By the mean-value theorem,
\begin{eqnarray*}
	\esp_n\left[
		\sup_{t\leq T} 
		\left|\int^t_0  \sqrt{n}\sum_j \beta^2 u_j(sn) \gradn \varphi^n_t ds
		- \int^t_0 \fieldn_s(\partial_x \varphi)\, ds
	\right|^2\right]
	\leq C \frac{T^2}{n}. 
\end{eqnarray*}
By tightness of the field, we then get
\begin{eqnarray*}
	\lim_{n\to\infty}\int^{\cdot}_0 \sqrt{n}\sum_j \beta^2 u_j(sn)\gradn \varphi^n_t ds
	=
	\int^{\cdot}_0 \field_s(\partial_x\varphi)\, ds.
\end{eqnarray*}
The convergence of the terms involving $W_j$ instead of $u_j$ follows by comparison as
\begin{eqnarray*}
	\esp_n\left[
		\sup_{t\leq T}
		\left|\int^t_0 \sqrt{n}\sum_j\left\{ \beta^2 u_j(sn) - \frac{1}{\sqrt{n}} W_j(sn) \right\} \gradn \varphi^n_t ds
	\right|^2\right]
	\leq C \frac{T^2}{\sqrt{n}}. 
\end{eqnarray*}
Hence, all linear terms appearing in the previous section converge to transport terms.

\subsection{Martingale term}
The quadratic variation of the martingale part satisfies
	\begin{eqnarray*}
		\lim_{n\to\infty} \langle\mart^n(\varphi)\rangle_t = t \| \partial_x \varphi \|^2_{L^2}.
	\end{eqnarray*}
	By a criterion of Aldous \cite{A}, this implies convergence to the white noise.
\subsection{Symmetric term}
Recall that
\begin{eqnarray*}
	\spart^n_t(\varphi)
	=
	\frac12 \int^t_0 \sum_j \bW_j(sn)\lapln \varphi^n_j ds.
\end{eqnarray*}
The argument used to treat the linear terms in Section \ref{sec:linear-terms} immediately shows that
\begin{eqnarray*}
	\lim_{n\to\infty} \spart^n_{\cdot}(\varphi) = \int^{\cdot}_0 \field_s(\partial_x^2 \varphi) ds.
\end{eqnarray*}
\subsection{Anti-symmetric term}
All that is left is to identify the limit of the term $\widetilde{\mathcal{B}}^n_t$. Define a modified version of the field by
\begin{eqnarray*}
	\widetilde{\mathcal{X}}^n_t(\varphi) = \sum_{j} \bW_j(sn)\gradn \varphi^n_j.
\end{eqnarray*}
By careful $L^2$ computations,
\begin{eqnarray*}
\esp\left[ \sup_{t\leq T}\left| \int^t_0 \left( \fieldn_r(\varphi)- \widetilde{\mathcal{X}}^n_r(\varphi)\right)\, dr \right|^2\right]
	&\leq&
	C \frac{T^2}{\sqrt{n}},\\
	\esp\left[ \sup_{t\leq T}\left| \int^t_0 \left( \fieldn_r(\varphi)^2- \widetilde{\mathcal{X}}^n_r(\varphi)^2\right)\, dr \right|^2\right]
	&\leq&
	C \frac{T^2}{\sqrt{n}},
\end{eqnarray*}
so that, when integrated over time, the field and the modified field (and their squares) are equivalent.

Recall $\iota_{\varepsilon}(x)=\varepsilon^{-1}{\bf 1}_{(x,x+\varepsilon]}$ and observe that
\begin{eqnarray*}
	\sqrt{n} \sum_j \tau_j Q(\varepsilon \sqrt{n}, nt) \grad^n \phi^n_j
	&=&
	\frac{1}{\sqrt{n}} \sum_j \left( \frac{1}{\varepsilon} \sum_k \bW_k {\bf 1}_{[\tfrac{j-nt}{\sqrt{n}},\tfrac{j-nt}{\sqrt{n}}+\varepsilon)}(\tfrac{k-nt}{\sqrt{n}}) \right)^2 \grad^n \phi^n_j \\
	&=&
	\frac{1}{\sqrt{n}} \sum_j \left( \widetilde{\mathcal{X}}^n_t(\iota_{\varepsilon}(\tfrac{j-nt}{\sqrt{n}}))\right)^2 \grad^n \phi^n_j.
\end{eqnarray*}
From here, we obtain the limit 
\begin{eqnarray*}
	\mathcal{A}^{\varepsilon}_{s,t}(\varphi) = \lim_{n\to\infty} \int^t_s\sqrt{n}\sum_j \tau_jQ(\varepsilon \sqrt{n}, nr) \gradn \varphi^n_j dr. 
\end{eqnarray*}
This does not follow immediately from the convergence of the field as $\iota_{\varepsilon}$ is not an $\mathcal{S}(\re)$ function. However, it can be approximated by $\mathcal{S}(\re)$ functions from where the convergence follows (see \cite{GJ-arma}, Section 5.3).

By Proposition \ref{thm:second-order},
\begin{eqnarray*}
	\esp\left[ \left| \widetilde{\mathcal{B}}_t(\varphi)-\widetilde{\mathcal{B}}_s(\varphi)
	- \int^t_s \sqrt{n}\sum_j \tau_j Q(l,nr)\gradn \varphi^n_j dr 
	\right|^2\right]
	&\leq&
	C \left( \frac{(t-s)l}{\sqrt{n}} + \frac{(t-s)^2}{l^2} \right).
\end{eqnarray*}
With $l \sim \varepsilon \sqrt{n}$ and taking the limit as $n\to\infty$,
\begin{eqnarray}\label{eq:B-A}
	\esp\left[ \left| \widetilde{\mathcal{B}}_t(\varphi)-\widetilde{\mathcal{B}}_s(\varphi)
	- \mathcal{A}^{\varepsilon}_{s,t}(\varphi) 
	\right|^2\right]
	&\leq&
	C (t-s) \varepsilon.
\end{eqnarray}
The crucial estimate (EC2) follows from the triangle inequality. By Theorem \ref{thm:S-EC-imply}, we get the existence of the limit
\begin{eqnarray*}
	\mathcal{A}_t(\varphi) = \lim_{\varepsilon\to0} \mathcal{A}^{\varepsilon}_{0,t}(\varphi).
\end{eqnarray*}
Estimate \eqref{eq:B-A} further yields $\widetilde{\mathcal{B}}=\mathcal{A}$.

We  now check that $\field$ satisfies the estimate (EC1). By \eqref{eq:B-A}, it is enough to check that
\begin{eqnarray*}
	\esp\left[ \left| \int^t_0\widetilde{\mathcal{X}}^n_s(\partial_x^2 \varphi)\, ds
	\right|^2\right]
	\leq
	\kappa t,
\end{eqnarray*}
for all $n\geq 1$.
By a summation-by-parts and the smoothness of $\varphi$, it is enough to check that
\begin{eqnarray}\label{eq:last-KV}
	\esp\left[ \left| 
		\int^t_0
		\sqrt{n} \sum_j \left( \bW_{j-1}(sn)-\bW_j(sn)\right) \gradn \varphi^n_j ds
	\right|^2\right]
	\leq \kappa t.
\end{eqnarray} 
This follows at once from Kipnis-Varadhan inequality and the following computation: with $f\in\mathscr{C}$, integration-by-parts and Young's inequality yield,
\begin{eqnarray*}
	&&2 \int \sqrt{n} \sum_j \left( \bW_{j-1}(sn)-\bW_j(sn)\right) \gradn \varphi^n_j f\, d\mu_n
	\\
	&=&
	2 \int \sum_j \gradn\varphi^n_j \left( \partial_{j-1}-\partial_j\right)   f\, d\mu_n \\
	&\leq&
	\frac{2}{\sqrt{n}}\sum_j (\gradn\varphi^n_j)^2  
	+ \frac{\sqrt{n}}{2} \int \sum_j \left( ( \partial_{j-1}-\partial_j)   f\right)^2 d\mu_n.
\end{eqnarray*}
This proves \eqref{eq:last-KV}.

Finally, we note that all our estimates can be applied to the reversed process $\{\field^n_{T-t}:\, t\in[0,T] \}$. This shows that $\field$ satisfies Condition 3 of Definition \ref{def:ES}.


\appendix

\section{Estimates on the terms of order 3}

The goal of this section is to estimate the term
\begin{eqnarray*}
	\int^t_0 \sqrt{n}\sum_j W_{j-1}(sn)u_j^2(sn)\gradn \varphi^n_j ds.
\end{eqnarray*}
We will show that this expression  only contributes a few transport terms. The following computations are inspired by \cite{BGJS}.

We start with the observation that, in monomials of order 3 (and higher), we can replace each instance of the $u_j$'s by the corresponding $W_j$'s by paying the price of a term that converges to zero in the ucp topology. For example, a simple $L^2$ computation shows that
\begin{eqnarray*}
	\esp_n\left[\sup_{t\leq T}\left| \int^t_0 \sqrt{n}\sum_j (W_{j-1}(sn)u_j^2(sn)-W_{j-1}(sn)W_j^2(sn)) \gradn \varphi^n_j ds\right|^2\right]
	&\leq& C\frac{T^2}{\sqrt{n}}.
\end{eqnarray*}
Much in the same way, we see that indexes can be shifted. For instance, by a summation by parts, we get
\begin{eqnarray*}
	\esp_n\left[\sup_{t\leq T}\left| \int^t_0 \sqrt{n}\sum_j (W_{j}^3(sn)-W_{j-1}^3(sn)) \gradn \varphi^n_j ds\right|^2\right]
	&\leq& C\frac{T^2}{n}.
\end{eqnarray*}
Any monomial of degree 3 can be treated similarly. Note that indexes can also be shifted in expressions involving $\beta^2 u_j$.

Based on these considerations,
\begin{eqnarray*}
	L u_{j-1}u_ju_{j+1} &=& W_{j-2}W_{j}W_{j+1} - 3W_{j-1}W_jW_{j+1} \\
									&&+ W_{j-1}^2W_{j+1}+W_{j-1}W_j^2
									+ E^{(1)}_j,\\
	Lu^2_{j-1}u_{j+1} &=& 2\beta^2u_{j} + 2W_{j-2}W_{j-1}W_{j+1}-3W_{j-1}^2W_{j+1}+W_{j-1}^2W_j+ E^{(2)}_j\\			
	Lu_{j-1}^2u_j &=&  2W_{j-2}W_{j-1}W_j-3W_{j-1}^2W_j + W_{j-1}^3 + E^{(3)}_j,\\
	Lu^3_j &=& 6\beta^2 u_j +3W_{j-1}W_j^2-3W_{j-1}^3 +  E^{(4)}_j,\\
\end{eqnarray*}
where
\begin{eqnarray*}
	\esp_n\left[\sup_{t\leq T}\left| \int^t_0 \sqrt{n}\sum_j E^{(i)}_j(sn) \gradn \varphi^n_j ds\right|^2\right]
	\leq C\frac{T^2}{\sqrt{n}}, \quad i=1,\cdots,4.
\end{eqnarray*}
Putting everything together, we get

\begin{eqnarray}
	\nonumber
	10 W_{j-1}W_j^2 &=& -8 \beta^2 u_j
	 + \frac13 Lu_j^3 + Lu_{j-1}^2u_j + 3Lu_{j-1}^2u_j + 9Lu_{j-1}u_ju_{j+1}\\
	 \nonumber
	 &&
	 - 2W_{j-2}W_{j-1}W_j
	 -6W_{j-2}W_{j-1}W_{j+1}
	 -9W_{j-2}W_jW_{j+1}\\
	 &&
	 \label{eq:decomposition-order3}
	 +27W_{j-1}W_jW_{j+1}+ E_j,
\end{eqnarray}
with
\begin{eqnarray*}
	\esp_n\left[\sup_{t\leq T}\left| \int^t_0 \sqrt{n}\sum_j E_j(sn) \gradn \varphi^n_j ds\right|^2\right]
	\leq C\frac{T^2}{\sqrt{n}}.
\end{eqnarray*}
By an $L^2$ computation, the linear terms are readily seen to be tight (and to contribute to transport terms in the limit). 
The terms involving $L$ can be treated with the method of Lemma \ref{thm:lemma-Lu2} and \ref{thm:lemma-Lu3} and vanish in the limit.
The rest of this section is devoted to show that the monomials of order 3 in $W$ can be neglected. It amounts to showing a second-order Boltzmann-Gibbs principle for these terms:
\begin{proposition}\label{thm:BG-3}
	Let $l\geq 2$. There exists a constant $C>0$ such that
	\begin{eqnarray*}
		&&\esp_n\left[ \sup_{t\leq T}\left| 
			\int^t_0 \sqrt{n} \sum_j \Big\{ \bW_{j-1}(sn)\bW_j(sn)\bW_{j+1}(sn) 
			\right. \right. \\
		&&\phantom{blablablablablablablablabla}
		\left. \left. 
		-\left( \overrightarrow{W}^l_{j+1}(sn)\right)^3\Big\} \gradn \varphi^n_j ds
		\right|^2\right] 
		\leq C \left\{ \frac{Tl}{n} +\frac{T^2}{l^2}\right\} \mathcal{E}(\partial_x \varphi).
	\end{eqnarray*}
\end{proposition}
\begin{proof}
	Let
	\begin{eqnarray*}
		F_{n,j} 
		&=&
		\left(\overrightarrow{W}^l_{j+1}\right)^2 \left[ \bW_{j-1}-\overrightarrow{W}^l_j\right]\\
		&&
		+ \frac{2\beta^2}{l}\overrightarrow{W}^l_{j+1}\left[\overrightarrow{W}^l_{j}-\bW_j\right]
		- \frac{2\beta^2(l-1)}{l^2}	\bW_j\overrightarrow{W}^l_{j+1}\\
		&&
		+\frac{2\sigma_n(l-1)}{l^2} \overrightarrow{W}^l_{j+1},
	\end{eqnarray*}
	and use the decomposition
	\begin{eqnarray*}\nonumber
		\bW_{j-1}\bW_j\bW_{j+1}-\left( \overrightarrow{W}^l_{j+1}\right)^3 
		&=&
		\bW_{j-1} \left\{ \bW_j\bW_{j+1}-\left( \overrightarrow{W}^l_{j+1}\right)^2+\frac{\sigma_n}{l}\right\}
		+ F_{n,j}		
		\\
		&-&
		\frac{\sigma_n}{l}\bW_{j-1}
		-\frac{2\beta^2}{l} \overrightarrow{W}^l_{j+1}\left( \overrightarrow{W}^l_j - \bW_j\right)
		\\
		 &+&
		 \frac{2\beta^2(l-1)}{l^2}\bW_j \overrightarrow{W}^l_{j+1}
		- \frac{2\sigma_n(l-1)}{l^2}	\overrightarrow{W}^l_{j+1} 
		 \\
		 &+&
		 \left( \overrightarrow{W}^l_{j+1}\right)^2
		 \left\{
		 	\overrightarrow{W}^l_{j}-\overrightarrow{W}^l_{j+1}
		 \right\}.
	\end{eqnarray*}
	The first term on the right-hand-side can be treated with a straightforward adaptation of Proposition \ref{thm:second-order} and gives the bound $\displaystyle C \left(\frac{Tl}{n}+\frac{T^2}{l^2\sqrt{n}}\right)$. 
	The second term is the object of the next lemma.
	The remaining terms can be handled by $L^2$ computations producing overall the bound  $CT^2/l^2$. 
\end{proof}
\begin{lemma}
	There exists a constant $C>0$ such that
	\begin{eqnarray*}
		&&\esp_n\left[ \sup_{t\leq T}\left| 
			\int^t_0 \sqrt{n} \sum_j  
		 		F_{n,j}(sn)
			\gradn \varphi^n_j ds
		\right|^2\right]
		\leq C \frac{T}{n}\mathcal{E}(\partial_x\varphi).
	\end{eqnarray*}
\end{lemma}
\begin{proof}
	Let $f\in\mathscr{C}$.
	Using integration-by-parts,
	\begin{eqnarray*}
		\int \sqrt{n}\sum_j  F_{n,j}(sn)	
		\varphi_j f\, d\mu_n
		=
		\int \sum_j \sum^{l-1}_{k=0} \psi_k
			\left( \overrightarrow{W}^l_{j+1}\right)^2 \left( \partial_{j+k-1}-\partial_{j+k}\right) f \varphi_j d\mu_n,
	\end{eqnarray*}
	where $\psi_k = (l-k)/l$. By Young's inequality, twice the above is bounded by
	\begin{eqnarray*}
		\int \sum_j \sum^{l-1}_{k=0} \psi_k \left\{
			\alpha\left( \overrightarrow{W}^l_{j+1}\right)^4 \varphi_j^2
			+\frac{1}{\alpha}\left( \left( \partial_{j+k-1}-\partial_{j+k}\right) f \right)^2 
			\right\}d\mu_n.
	\end{eqnarray*}
	Taking $\alpha=2l/\sqrt{n}$, this is further bounded by
	\begin{eqnarray*}
		&&2\frac{l^2}{\sqrt{n}}\sum_j  \varphi_j^2 \int \left( \overrightarrow{W}^l_{j+1}\right)^4d\mu_n 
		 + \frac{\sqrt{n}}{2} \int \sum_j\left( \left( \partial_{j+k-1}-\partial_{j+k}\right) f \right)^2d\mu_n\\
		 && \leq
		 \frac{2C}{n^{3/2}} \sum_j \varphi_j^2 + \frac{\sqrt{n}}{2} \int \sum_j\left( \left( \partial_{j+k-1}-\partial_{j+k}\right) f \right)^2d\mu_n.
	\end{eqnarray*}
	The result follows from Kipnis-Varadhan inequality.
\end{proof}
For $T\geq 1/n$ and $l\sim \sqrt{Tn}$, the estimate in Proposition \ref{thm:BG-3} becomes $CT^{3/2}/\sqrt{n}$. On the other hand, a careful $L^2$ computation taking dependencies into account yields
\begin{eqnarray*}
		&&\esp_n\left[ \sup_{t\leq T}\left| 
			\int^t_0 \sqrt{n} \sum_j 		
			\left( \overrightarrow{W}^l_{j+1}(sn)\right)^3
			\gradn \varphi^n_j ds
		\right|^2\right] 
		\leq C \frac{T^2}{l^2}
		\leq 
	C\frac{T^{3/2}}{\sqrt{n}}.
	\end{eqnarray*}
	Combining both estimates:
\begin{eqnarray}\nonumber
	\esp\left[ \sup_{t\leq T}\left| \int^t \sqrt{n}\sum_j \bW_{j-1}(sn)\bW_j(sn)\bW_{j+1}(sn) \gradn \varphi^n_j ds\right|^2\right]
	\leq
	C\frac{T^{3/2}}{\sqrt{n}},
\end{eqnarray}
for $T\geq 1/n$. For $T<1/n$, an $L^2$ computation yields 
\begin{eqnarray}\nonumber
	\esp\left[ \sup_{t\leq T}\left| \int^t \sqrt{n}\sum_j \bW_{j-1}(sn)\bW_j(sn)\bW_{j+1}(sn) \gradn \varphi^n_j ds\right|^2\right]
	\leq
	C T^2
	\leq 
	C\frac{T^{3/2}}{\sqrt{n}}.
\end{eqnarray}
This allows us to control the centered monomials of order 3. To remove the centering, note that difference $\bW_{j-1}\bW_j\bW_{j+1}-W_{j-1}W_jW_{j+1}$ only involves terms of the form 
\begin{eqnarray*}
	\frac{1}{\sqrt{n}}W_kW_m \quad \text{and} \quad \frac{1}{n} W_k,
\end{eqnarray*}
where we used $\esp_n[W]=O(1/\sqrt{n})$. As we know that the terms of order two are tight, the extra $1/\sqrt{n}$ above makes them to vanish. Finally, the linear terms are tight when normalized by $\sqrt{n}$ and hence vanish with the current normalization.

The other terms of order 3 in \eqref{eq:decomposition-order3} are treated similarly.


\begin{thebibliography}{99}
\addcontentsline{toc}{chapter}{Bibliography}

\bibitem{AS} Abramowitz, M. and Stegun, I. A. (1992) \textit{Handbook of Mathematical Functions with Formulas, Graphs, and Mathematical Tables}, Dover, New York.

\bibitem{AKQ} Alberts, T., Khanin, K. and Quastel, J. (2014)
\textit{ The intermediate disorder regime for directed polymers in dimension 1 + 1}, 
Ann. Probab. 42, 1212–1256

\bibitem{AKQ2} Alberts, T, Khanin K. and Quastel (2014)
\textit{The continuum directed random polymer. } J. Stat. Phys. 154, (1), 154- 305

\bibitem{A} Aldous, D. (1981) 
\textit{Weak convergence and the general theory of processes},
Unpublished notes


%
\bibitem{ACQ} Amir, G., Corwin, I. and Quastel, J. (2010)
\textit{Probability distribution of the free energy of the continuum directed random polymer in 1 + 1 dimensions }, 
Comm. Pure. Appl. Math. 64, (4), 466- 537
%

%
\bibitem{BGJS} Bernardin, C., Goncalves, P., Jara, M. and Simon, M. (2018) \textit{Nonlinear Perturbation of a Noisy Hamiltonian Lattice Field Model: Universality Persistence}, Comm. Math. Phys. 361, 2, 605-659 

\bibitem{BG} Bertini, L. and  Giacomin, G. (1997)
\textit{Stochastic Burgers and KPZ equations from particle systems}, 
Comm. Math. Phys.  183, (3), 571- 607
%



\bibitem{CM} Cannizzaro, G., and Matetski, K. (2018) \textit{Space–Time Discrete KPZ Equation}, Comm. Math. Phys. 358, 2, 521-588

\bibitem{CC} Catellier, R. and Chouk, K. (2018) \textit{Paracontrolled distributions and the 3-dimensional stochastic quantization equation}, Ann. Probab. 46, 5, 2621–2679

\bibitem{CGP} Chouk, K., Gairing, J. and Perkowski, N. (2017) \textit{An invariance principle for the two-dimensional parabolic Anderson model with small potential}, Stoch. Partial Differ. Equ. Anal. Comput. 5, no. 4, 520–558

\bibitem{Comets-book} Comets, F. (2017) \textit{Directed Polymers in Random Environment: \'Ecole d'\'Et\'e de Probabilit\'es de Saint-Flour XLVI 2016}, Lecture Notes in Mathematics 2175, Springer


\bibitem{C-review} Corwin, I. (2012)
\textit{The Kardar-Parisi-Zhang equation and universality class}, 
Random Matrices Theory Appl. 1,1130001

\bibitem{CG} Corwin, I and Gu, Y. (2017) \textit{Kardar–Parisi–Zhang Equation and Large Deviations for Random Walks in Weak Random Environments}, J. Stat. Phys. 166, 1, 150-168



%
\bibitem{DGP} Diehl, J., Gubinelli, M. and Perkowski, N. (2016)
\textit{The Kardar-Parisi-Zhang equation as scaling limit of weakly asymmetric interacting Brownian motions},
 Comm. Math. Phys. 354, no. 2, 549-589
%

\bibitem{FS} Ferrari, P. L. and Spohn, H. (2011) \textit{Random growth models}, in The Oxford handbook of random
matrix theory, 782–801, Oxford Univ. Press, Oxford, 2011.

\bibitem{Gartner} Gartner, J. (1988) \textit{Convergence towards Burgers’ equation and propagation of chaos for weakly asymmetric exclusion processes}, Stoch. Proc. and Appl. 27, 233-260



%
\bibitem{GJ-arma} Goncalves, P. and Jara, M. (2014)
\textit{Nonlinear fluctuations of weakly asymmetric interacting particle systems},
Arch. Ration. Mech. Anal. 212, no. 2, 597-644

\bibitem{GJSeth} Goncalves, P., Jara, M. and Sethuraman, S. (2015) \textit{A stochastic Burgers equation from a class of microscopic interactions}, Ann. Probab.
43, 1, 286-338.



%
\bibitem{GJS} Goncalves, P., Jara, M. and Simon, M. (2017)
\textit{Second order Boltzmann-Gibbs principle for polynomial functions and applications},
J. Stat. Phys. 166, no. 1, 90–113
%
\bibitem{GuJ} Gubinelli, M. and Jara, M. (2013)
\textit{Regularization by noise and stochastic Burgers equations},
Stoch. Partial Differ. Equ. Anal. Comput. 1, no. 2, 325–350

\bibitem{GP-HQ}  Gubinelli, M. and Perkowski, N. (2016) \textit{The Hairer-Quastel universality result in equilibrium}, 
In: Stochastic analysis on large scale interacting systems, 101–115, RIMS K\^oky\^uroku Bessatsu, B59, Res. Inst. Math. Sci. (RIMS), Kyoto, 2016

%
\bibitem{GP-reloaded} Gubinelli, M. and Perkowski, N. (2017)
\textit{KPZ reloaded},
Comm. Math. Phys. 349, no. 1, 165–269
%
\bibitem{GP-uniqueness} Gubinelli, M. and Perkowski, N. (2018)
\textit{Energy solutions of KPZ are unique},
 J. Amer. Math. Soc. 31, 427-471
%
\bibitem{GP-noneq} Gubinelli, M. and Perkowski, N. (2018)
\textit{Probabilistic approach to the stochastic Burgers equation},
In: Eberle A., Grothaus M., Hoh W., Kassmann M., Stannat W., Trutnau G. (eds) Stochastic Partial Differential Equations and Related Fields. SPDERF 2016. Springer Proceedings in Mathematics \& Statistics, vol 229. Springer, Cham

\bibitem{GP-generator} Gubinelli, M. and Perkowski, N. (2018) \textit{The infinitesimal generator of the stochastic Burgers equation}, preprint, arXiv:1810.12014

%
\bibitem{H-KPZ} Hairer, M. (2013) 
\textit{Solving the KPZ equation},
Annals of Mathematics 178, 559–664

\bibitem{H-RS} Hairer, M. (2014) \textit{A theory of regularity structures}, Inv. Math. 198, 2, 269-504




\bibitem{HM} Hairer, M. and Matetski, K. (2018) \textit{Discretisations of rough stochastic PDEs}, Ann. Probab. 46, 3, 1651-1709

\bibitem{HQ} Hairer, M. and Quastel, J. (2018) \textit{A class of growth models rescaling to KPZ}, Forum Math. Pi 6, e3, 112 pp.


\bibitem{HX} Hairer, M. and Xu, W. (2019) \textit{Large-scale limit of interface fluctuation models}, to appear in Ann. of Probab.

 \bibitem{Henley-Huse}
 D. Henley and C. Huse (1985)
\textit{ Pinning and roughening of domain wall in Ising systems due to random impurities}, Phys. Rev. Lett. 54, 2708--2711

\bibitem{JM} Jara, M. and Moreno Flores, G. (2019) \textit{Scaling of the Sasamoto-Spohn model in equilibrium},
Electron. Commun. Probab., 24, paper no. 3, 12 pp.


%
\bibitem{KPZ} Kardar, M., Parisi, G and Zhang, Y--C. (1986)
\textit{Dynamic scaling of growing interfaces} 
Phys. Rev. Lett., 56(9):889–892
%
%
%

\bibitem{L} Labb\'e, C (2018) \textit{On the scaling limits of weakly asymmetric bridges},
Probab. Surveys 15, 156-242.




%
%

\bibitem{MP} Martin, J. and Perkowski, N. (2017) \textit{Paracontrolled distributions on Bravais lattices and weak universality of the 2d parabolic Anderson model}, preprint, arXiv:1704.08653

\bibitem{MO} Menz, G. and Otto, F. (2013) \textit{Uniform logarithmic sobolev inequalities for conservative spin systems with super-quadratic single-site potential}, Ann. Probab. 41, 3, 208-238


\bibitem{Mitoma} Mitoma, I. (1983)
\textit{Tightness of probabilities in $C([0,1],Y')$ and $D([0,1],Y')$},
Ann. Probab., 11, 4, 989-999
%
\bibitem{MQR} Moreno Flores, G., Quastel, J. and Remenik, D., 
\textit{in preparation}

\bibitem{MSV} Moreno Flores, G., Seppalainen, T. and Valko, B. (2014)
\textit{Fluctuation exponents for directed polymers in the intermediate disorder regime. } Electron. J. Probab. 19, (89), 28 pp, (2014)

\bibitem{OY} O'Connell, N. and Yor. M. (2001)
\textit{Brownian analogues of Burke's theorem. } Stochastic Process. Appl. 96, (2), 285- 304

\bibitem{PR} Perkowski, N. and Rosati, T. C. (2018) \textit{The KPZ equation on the real line}, preprint, 	arXiv:1808.00354

\bibitem{SS} Sasamoto, T. and Spohn, H. (2009)
\textit{Superdiffusivity of the 1D Lattice Kardar-Parisi-Zhang Equation},
J. Stat. Phys., 137: 917–935

\bibitem{SV} Seppalainen, T., Valko, B. (2010) \textit{Bounds for scaling exponents for a 1+1 dimensional directed polymer in a Brownian environment}, Alea 7, 451-476

\bibitem{S} Spohn, H. (2014) \textit{KPZ scaling theory and the semidiscrete
directed polymer model}, Random Matrices, MSRI Publications, Vol. 65


\bibitem{SQ} Spohn, H. and Quastel, J. (2015) \textit{The One-Dimensional KPZ Equation and Its Universality Class}, J. Stat. Phys. 160,  4, 965-984


\bibitem{liquid} Takeuchi, K. and Sano, M. (2012) \textit{Evidence for Geometry-Dependent Universal Fluctuations of the Kardar-Parisi-Zhang Interfaces in Liquid-Crystal Turbulence}, J. Stat. Phys. 147, 5, 853-890




%
\bibitem{Q-review} Quastel, J. (2012)
\textit{Introduction to KPZ},
 Curr. Dev. Math. 2011, 125–194, Int. Press, Somerville, MA
%





\end{thebibliography}
\end{document}